\documentclass[12pt]{article}

\usepackage[T1]{fontenc}
\usepackage[utf8]{inputenc}
\usepackage[fleqn]{amsmath}
\usepackage{amsthm,amssymb}
\usepackage{graphicx}
\usepackage[colorlinks=true,citecolor=black,linkcolor=black,urlcolor=blue]{hyperref}

\usepackage{tikz}
\usetikzlibrary{calc}

\usepackage{color}
\usepackage{enumerate}
\usepackage{epsfig}
\usepackage{pgf}
\usepackage{pgfcore}
\usepackage{footnote}

\usepackage{diagbox}

\usepackage{ytableau}

\theoremstyle{plain}
\newtheorem{theorem}{Theorem}[section]

\newtheorem{corollary}[theorem]{Corollary}
\newtheorem{proposition}[theorem]{Proposition}

\theoremstyle{definition}

\newtheorem{example}[theorem]{Example}
\newtheorem{conjecture}[theorem]{Conjecture}

\theoremstyle{remark}

\tikzstyle{noeud}=[circle,inner sep=2, minimum size =3 pt, line width = 1pt, draw=black, fill=white]

\newcommand{\Z}{{\mathbb{Z}}}

\newcommand{\sg}{{\rm sg}}

\begin{document}

\title{Strong geodetic problem on complete multipartite graphs}

\author{
	Vesna Ir\v si\v c $^{a,b}$
	\and
	Matja\v z Konvalinka $^{a,b}$
}

\date{\today}

\maketitle
\begin{center}
	$^a$ Faculty of Mathematics and Physics, University of Ljubljana, Slovenia\\
	\medskip

	$^b$ Institute of Mathematics, Physics and Mechanics, Ljubljana, Slovenia
\end{center}

\begin{abstract}
The strong geodetic problem is to find the smallest number of vertices such that by fixing one shortest path between each pair, all vertices of the graph are covered. In this paper we study the strong geodetic problem on complete bipartite graphs; in particular, we discuss its asymptotic behavior. Some results for complete multipartite graphs are also derived. Finally, we prove that the strong geodetic problem restricted to (general) bipartite graphs is NP-complete.
\end{abstract}


\medskip\noindent
{\bf AMS Subj.\ Class:} 05C12, 05C70; 68Q17

\section{Introduction}
\label{sec:intro}

The strong geodetic problem was introduced in~\cite{MaKl16a} as follows. Let $G=(V,E)$ be a graph. For a set $S\subseteq V$, and for each pair of vertices $\{x,y\}\subseteq S$, $x\ne y$, define $\widetilde{g}(x,y)$ as a {\em selected fixed} shortest path between $x$ and $y$. We set
$$\widetilde{I}(S)=\{\widetilde{g}(x, y) : x, y\in S\}\,,$$ and $V(\widetilde{I}(S))=\bigcup_{\widetilde{P} \in \widetilde{I}(S)} V(\widetilde{P})$. If $V(\widetilde{I}(S)) = V$ for some $\widetilde{I}(S)$, then the set $S$ is called a {\em strong geodetic set}. This means that the selected fixed geodesics between vertices from $S$ cover all vertices of the graph $G$. If $G$ has just one vertex, then its vertex is considered the unique strong geodetic set. The {\em strong geodetic problem} is to find a minimum strong geodetic set of $G$. The size of a minimum strong geodetic set is the {\em strong geodetic number} of $G$ and is denoted by $\sg(G)$. A strong geodetic set of size $\sg(G)$ is also called an {\em optimal strong geodetic set}.

In the first paper~\cite{MaKl16a}, it was proved that the problem is NP-complete. The invariant has also been determined for complete Apollonian networks~\cite{MaKl16a}, thin grids and cylinders~\cite{Klavzar+2017}, and balanced complete bipartite graphs~\cite{bipartite}. Some properties of the strong geodetic number of Cartesian product of graphs have been studied in~\cite{products}. Recently, a concept of strong geodetic cores has been introduced and applied to the Cartesian product graphs~\cite{cores}. An edge version of the problem was defined and studied in~\cite{MaKl16b}.

The strong geodetic problem is just one of the problems which aim to cover all vertices of a graph with shortest paths. Another such problem is the \emph{geodetic problem}, in which we determine the smallest number of vertices such that the geodesics between them cover all vertices of the graph~\cite{bresar-2011b, chartrand-2000, hansen-2007, HLT93}. Note that we may use more geodesics between the same pair of vertices. Thus this problem seems less complex than the strong geodetic problem. It is known to be NP-complete on general graphs~\cite{atici}, on chordal and bipartite weakly chordal graphs~\cite{dourado}, on co-bipartite graphs~\cite{ekim-2014}, and on graphs with maximal degree $3$~\cite{bueno}. However, it is polynomial on co-graphs and split graphs~\cite{dourado}, on proper interval graphs~\cite{ekim-2012}, on block-cactus graphs and monopolar chordal graphs~\cite{ekim-2014}. Moreover, the geodetic number of complete bipartite (and multipartite) graphs is straightforward to determine, i.e.\ $\sg(K_{n,m}) = \min\{n,m,4\}$~\cite{HLT93}.

Recall from~\cite{bipartite} that the strong geodetic problem on a complete bipartite graph can be presented as a (non-linear) optimization problem as follows. Let $(X, Y)$ be the bipartition of $K_{n_1, n_2}$ and $S = S_1 \cup S_2$, $S_1 \subseteq X$, $S_2 \subseteq Y$, its strong geodetic set.  Let $|S_i| = s_i$ for $i \in \{1,2\}$. Thus, $\sg(K_{n_1, n_2}) \leq s_1 + s_2$. With geodesics between vertices from $S_1$ we wish to cover vertices in $Y - S_2$. Vice versa, with geodesics between vertices from $S_2$ we are covering vertices in $X - S_1$. The optimization problem for $\sg(K_{n_1, n_2})$ reads as follows:
\begin{align}
\label{optimizationProblem}
\begin{split}
\min \quad & s_1 + s_2 \\
\text{subject to: } & 0 \leq s_1 \leq n_1\\
&  0 \leq s_2 \leq n_2\\
&  \binom{s_2}{2} \geq n_1 - s_1\\
&  \binom{s_1}{2} \geq n_2 - s_2\\
& s_1, s_2 \in \Z.
\end{split}
\end{align}

This holds due to the fact that every geodesic in a complete bipartite graph is either of length $0$, $1$ (an edge), or $2$ (a path with both endvertices in the same part of the bipartition). If a strong geodetic set $S$ has $k$ vertices in one part of the bipartition, then geodesics between those vertices can cover at most $\binom{k}{2}$ vertices in the other part.

The exact value is known for balanced complete bipartite graphs: if $n \geq 6$, then
$$\sg(K_{n,n}) = \begin{cases}
2 \left \lceil \displaystyle \frac{-1 + \sqrt{8 n + 1}}{2} \right \rceil, & 8n - 7 \text{ is not a perfect square},\vspace{0.2cm}\\
2 \left \lceil \displaystyle \frac{-1 + \sqrt{8 n + 1}}{2} \right \rceil - 1, & 8n - 7 \text{ is a perfect square}.
\end{cases}$$

See \cite[Theorem 2.1]{bipartite}.

In the following section, we generalize the above result to all complete bipartite graphs. In particular, we determine the asymptotic behavior of $\sg(K_{n,m})$.

As we also consider complete multipartite graphs, it is useful to recall the notation $\langle 1^{m_1}, \ldots, k^{m_k} \rangle$ which describes a partition with $m_i$ parts of size $i$, $1 \leq i \leq k$.

To conclude the introduction, we state the following interesting and surprisingly important fact.

\begin{proposition}
	\label{prop:two}
	For every complete multipartite graph there exist an optimal strong geodetic set such that its intersection with all but two parts of the multipartition is either empty or the whole part.
\end{proposition}

\proof
Let $G = K_{n_1, \ldots, n_r}$ be a complete multipartite graph with the multipartition $X_1, \ldots, X_r$, $|X_i| = n_i$ for all $i \in [r]$. Let $S = S_1 \cup \cdots \cup S_r$ be an optimal strong geodetic set, $S_i \subseteq X_i$ and $|S_i| = s_i$ for $i \in [r]$.

Suppose that $s_i \in \{1, \ldots, n_i - 1\}$ for $i \in \{1,2,3\}$. Without loss of generality, $s_1 = \min\{ s_1, s_2, s_3 \}$. Let $x \in S_1$, $y \in X_2 - S_2$, and $z \in X_3 - S_3$.

Define
\begin{align*}
	T_2 & = (S_1 - \{x\}) \cup (S_2 \cup \{y\}) \cup S_3 \cup \cdots \cup S_r, \\
	T_3 & = (S_1 - \{x\}) \cup S_2 \cup (S_3 \cup \{z\}) \cup S_4 \cup \cdots \cup S_r.
\end{align*}
Notice that $|S| = |T_2| = |T_3| = \sg(G)$. We now show that either $T_2$ or $T_3$ is a strong geodetic set of $G$.

Let $B_i$, $i \in \{2,3\}$, be the vertices in $X_i$ that are covered by the geodesics between vertices in $S_1$, and let $C$ be the vertices in $X_4 \cup \cdots \cup X_r$ that are covered by the geodesics between vertices in $S_1$. Geodesics between vertices in $S_1 - \{x\}$ cover $s_1 - 1$ vertices fewer than geodesics between vertices in $S_1$.

If the remaining geodesics can be rearranged so that $B_i$ is completely covered, then $T_i$ is a strong geodetic set, as the remaining uncovered vertices lie outside of $X_i$, so they can be covered by $s_i \geq s_1$ geodesics between $S_i$ and $y$ or $z$, $i \in \{2,3\}$.

If $|B_i \cup C| \geq s_1 - 1$, these
 geodesics can be rearranged so that $B_{5-i}$ is completely covered. If $|B_i| \leq 1$, we can also consider $B_i$ as completely covered (the possibly uncovered vertex can be only $y$ or $z$).

The only remaining case is if $|B_i \cup C| < s_1 - 1$, and $|B_i| \geq 2$ for $i \in \{2,3\}$. Since $|B_2 \cup B_3 \cup C| = \binom{s_1}{2}$, the first condition implies $s_1 < 4$, which implies $|B_2 \cup B_3| \leq \binom{s_1}{2}\leq 3$, so this cannot occur at all.

This means that every strong geodetic set with three or more parts of size unequal to $0$ or $n_i$ can be transformed into a strong geodetic set of the same size, where one of these parts becomes smaller and one larger. After repeating this procedure on other such triples, at most two parts can have size different from $0$ or $n_i$.
\qed

The rest of the paper is organized as follows. In the next section, some further results about the strong geodetic number of complete bipartite graphs are obtained. In Section~\ref{sec:complete_multipartite} we discuss the strong geodetic problem on complete multipartite graphs. Finally, in Section~\ref{sec:NP} the complexity of the strong geodetic problem on multipartite and complete multipartite graphs is discussed.

\section{On complete bipartite graphs}
\label{sec:complete_bipartite}

In this section, we give a complete description of the strong geodetic number of a complete bipartite graph. The result is unusual, however, as it does not directly say how to compute $\sg(K_{n,m})$; instead, it classifies triples $(n,m,k)$ for which $\sg(K_{n,m}) = k$.

Define $$f(\alpha, \beta) = \alpha - 1 + \binom{\max\{\beta-1, 2\}}{2} \, .$$

\begin{theorem}
	\label{thm:sg=k}
	For positive integers $n$, $m$ and $k$, $\sg(K_{n,m}) = k$ if and only if
	$$n < k \; \& \; m = f(k, n) \quad \text{or} \quad m < k \; \& \; n = f(k, m) \quad \text{or}$$
	$$f(k, i-1) \leq m \leq f(k, i) \; \& \; f(k, k-i-1) \leq n \leq f(k, k-i) \; \text{for some }i, 0 \leq i \leq k.$$
	The only exceptions are $\sg(K_{1,1}) = 2$ and $\sg(K_{2,2}) = 3$.
\end{theorem}

\begin{example}
The strong geodetic numbers of small complete bipartite graphs can be found in Table~\ref{tbl:values2}.

	\begin{table}[!!h]
		\centering
		\begin{tabular}{|c||*{15}{c|}}\hline
			\backslashbox{$m$}{$n$}
			&1&2&3&4&5&6&7&8&9&10&11&12&13&14&15 \\\hline\hline
			1 & 2 & 2 & 3 & 4 & 5 & 6 & 7 & 8 & 9 & 10 & 11 & 12 & 13 & 14 & 15 \\\hline
 2 & 2 & 3 & 3 & 4 & 5 & 6 & 7 & 8 & 9 & 10 & 11 & 12 & 13 & 14 & 15 \\\hline
 3 & 3 & 3 & 3 & 4 & 5 & 6 & 7 & 8 & 9 & 10 & 11 & 12 & 13 & 14 & 15 \\\hline
 4 & 4 & 4 & 4 & 4 & 4 & 4 & 5 & 6 & 7 & 8 & 9 & 10 & 11 & 12 & 13 \\\hline
 5 & 5 & 5 & 5 & 4 & 5 & 5 & 5 & 5 & 5 & 5 & 6 & 7 & 8 & 9 & 10 \\\hline
 6 & 6 & 6 & 6 & 4 & 5 & 6 & 6 & 6 & 6 & 6 & 6 & 6 & 6 & 6 & 6 \\\hline
 7 & 7 & 7 & 7 & 5 & 5 & 6 & 7 & 7 & 7 & 7 & 7 & 7 & 7 & 7 & 7 \\\hline
 8 & 8 & 8 & 8 & 6 & 5 & 6 & 7 & 8 & 8 & 8 & 8 & 8 & 8 & 8 & 8 \\\hline
 9 & 9 & 9 & 9 & 7 & 5 & 6 & 7 & 8 & 8 & 8 & 9 & 9 & 9 & 9 & 9 \\\hline
 10 & 10 & 10 & 10 & 8 & 5 & 6 & 7 & 8 & 8 & 8 & 9 & 9 & 9 & 9 & 10 \\\hline
 11 & 11 & 11 & 11 & 9 & 6 & 6 & 7 & 8 & 9 & 9 & 9 & 9 & 9 & 9 & 10 \\\hline
 12 & 12 & 12 & 12 & 10 & 7 & 6 & 7 & 8 & 9 & 9 & 9 & 10 & 10 & 10 & 10 \\\hline
 13 & 13 & 13 & 13 & 11 & 8 & 6 & 7 & 8 & 9 & 9 & 9 & 10 & 10 & 10 & 10 \\\hline
 14 & 14 & 14 & 14 & 12 & 9 & 6 & 7 & 8 & 9 & 9 & 9 & 10 & 10 & 10 & 10 \\\hline
 15 & 15 & 15 & 15 & 13 & 10 & 6 & 7 & 8 & 9 & 10 & 10 & 10 & 10 & 10 & 10 \\\hline
		\end{tabular}
		\caption{The strong geodetic numbers $\sg(K_{n,m})$ for some small complete bipartite graphs.}
		\label{tbl:values2}
	\end{table}

Figure \ref{fig:k12} shows the positions of all 201 pairs $(n,m)$ for which $\sg(K_{n,m}) = 12$. We can notice the ``parabolas'' corresponding to $m = f(k, n)$ and $n = f(k, m)$, as well as the ``intersecting rectangles'' corresponding to $f(k, i-1) \leq m \leq f(k, i)$, $f(k, k-i-1) \leq n \leq f(k, k-i)$.

\begin{figure}[!!h]
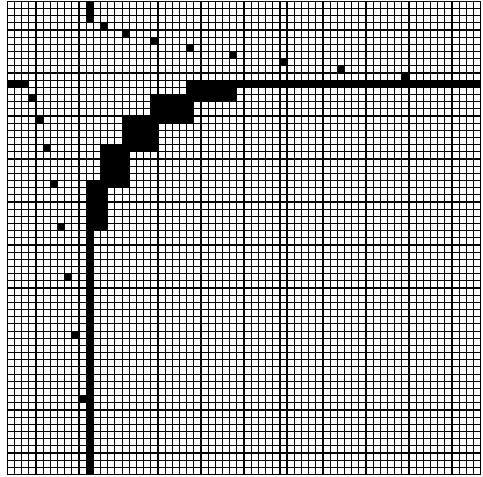

\begin{center}
	\scalebox{0.15}{\begin{ytableau}
	{} & {} & {} & {} & {} & {} & {} & {} & {} & {} & {} & *(black) & {} &
   {} & {} & {} & {} & {} & {} & {} & {} & {} & {} & {} & {} & {} &
   {} & {} & {} & {} & {} & {} & {} & {} & {} & {} & {} & {} & {} &
   {} & {} & {} & {} & {} & {} & {} & {} & {} & {} & {} & {} & {} &
   {} & {} & {} & {} & {} & {} & {} & {} & {} & {} & {} & {} & {} &
   {} \\
 {} & {} & {} & {} & {} & {} & {} & {} & {} & {} & {} & *(black) & {} &
   {} & {} & {} & {} & {} & {} & {} & {} & {} & {} & {} & {} & {} &
   {} & {} & {} & {} & {} & {} & {} & {} & {} & {} & {} & {} & {} &
   {} & {} & {} & {} & {} & {} & {} & {} & {} & {} & {} & {} & {} &
   {} & {} & {} & {} & {} & {} & {} & {} & {} & {} & {} & {} & {} &
   {} \\
 {} & {} & {} & {} & {} & {} & {} & {} & {} & {} & {} & *(black) & {} &
   {} & {} & {} & {} & {} & {} & {} & {} & {} & {} & {} & {} & {} &
   {} & {} & {} & {} & {} & {} & {} & {} & {} & {} & {} & {} & {} &
   {} & {} & {} & {} & {} & {} & {} & {} & {} & {} & {} & {} & {} &
   {} & {} & {} & {} & {} & {} & {} & {} & {} & {} & {} & {} & {} &
   {} \\
 {} & {} & {} & {} & {} & {} & {} & {} & {} & {} & {} & {} & {} &
   *(black) & {} & {} & {} & {} & {} & {} & {} & {} & {} & {} & {} & {} &
   {} & {} & {} & {} & {} & {} & {} & {} & {} & {} & {} & {} & {} &
   {} & {} & {} & {} & {} & {} & {} & {} & {} & {} & {} & {} & {} &
   {} & {} & {} & {} & {} & {} & {} & {} & {} & {} & {} & {} & {} &
   {} \\
 {} & {} & {} & {} & {} & {} & {} & {} & {} & {} & {} & {} & {} &
   {} & {} & {} & *(black) & {} & {} & {} & {} & {} & {} & {} & {} & {} &
   {} & {} & {} & {} & {} & {} & {} & {} & {} & {} & {} & {} & {} &
   {} & {} & {} & {} & {} & {} & {} & {} & {} & {} & {} & {} & {} &
   {} & {} & {} & {} & {} & {} & {} & {} & {} & {} & {} & {} & {} &
   {} \\
 {} & {} & {} & {} & {} & {} & {} & {} & {} & {} & {} & {} & {} &
   {} & {} & {} & {} & {} & {} & {} & *(black) & {} & {} & {} & {} & {} &
   {} & {} & {} & {} & {} & {} & {} & {} & {} & {} & {} & {} & {} &
   {} & {} & {} & {} & {} & {} & {} & {} & {} & {} & {} & {} & {} &
   {} & {} & {} & {} & {} & {} & {} & {} & {} & {} & {} & {} & {} &
   {} \\
 {} & {} & {} & {} & {} & {} & {} & {} & {} & {} & {} & {} & {} &
   {} & {} & {} & {} & {} & {} & {} & {} & {} & {} & {} & {} & *(black) &
   {} & {} & {} & {} & {} & {} & {} & {} & {} & {} & {} & {} & {} &
   {} & {} & {} & {} & {} & {} & {} & {} & {} & {} & {} & {} & {} &
   {} & {} & {} & {} & {} & {} & {} & {} & {} & {} & {} & {} & {} &
   {} \\
 {} & {} & {} & {} & {} & {} & {} & {} & {} & {} & {} & {} & {} &
   {} & {} & {} & {} & {} & {} & {} & {} & {} & {} & {} & {} & {} &
   {} & {} & {} & {} & {} & *(black) & {} & {} & {} & {} & {} & {} & {} &
   {} & {} & {} & {} & {} & {} & {} & {} & {} & {} & {} & {} & {} &
   {} & {} & {} & {} & {} & {} & {} & {} & {} & {} & {} & {} & {} &
   {} \\
 {} & {} & {} & {} & {} & {} & {} & {} & {} & {} & {} & {} & {} &
   {} & {} & {} & {} & {} & {} & {} & {} & {} & {} & {} & {} & {} &
   {} & {} & {} & {} & {} & {} & {} & {} & {} & {} & {} & {} & *(black) &
   {} & {} & {} & {} & {} & {} & {} & {} & {} & {} & {} & {} & {} &
   {} & {} & {} & {} & {} & {} & {} & {} & {} & {} & {} & {} & {} &
   {} \\
 {} & {} & {} & {} & {} & {} & {} & {} & {} & {} & {} & {} & {} &
   {} & {} & {} & {} & {} & {} & {} & {} & {} & {} & {} & {} & {} &
   {} & {} & {} & {} & {} & {} & {} & {} & {} & {} & {} & {} & {} &
   {} & {} & {} & {} & {} & {} & {} & *(black) & {} & {} & {} & {} & {} &
   {} & {} & {} & {} & {} & {} & {} & {} & {} & {} & {} & {} & {} &
   {} \\
 {} & {} & {} & {} & {} & {} & {} & {} & {} & {} & {} & {} & {} &
   {} & {} & {} & {} & {} & {} & {} & {} & {} & {} & {} & {} & {} &
   {} & {} & {} & {} & {} & {} & {} & {} & {} & {} & {} & {} & {} &
   {} & {} & {} & {} & {} & {} & {} & {} & {} & {} & {} & {} & {} &
   {} & {} & {} & *(black) & {} & {} & {} & {} & {} & {} & {} & {} & {} &
   {} \\
 *(black) & *(black) & *(black) & {} & {} & {} & {} & {} & {} & {} & {} & {} &
   {} & {} & {} & {} & {} & {} & {} & {} & {} & {} & {} & {} & {} &
   *(black) & *(black) & *(black) & *(black) & *(black) & *(black) & *(black) & *(black) & *(black) & *(black) &
   *(black) & *(black) & *(black) & *(black) & *(black) & *(black) & *(black) & *(black) & *(black) & *(black) &
   *(black) & *(black) & *(black) & *(black) & *(black) & *(black) & *(black) & *(black) & *(black) & *(black) &
   *(black) & *(black) & *(black) & *(black) & *(black) & *(black) & *(black) & *(black) & *(black) & *(black) &
   *(black) \\
 {} & {} & {} & {} & {} & {} & {} & {} & {} & {} & {} & {} & {} &
   {} & {} & {} & {} & {} & {} & {} & {} & {} & {} & {} & {} & *(black) &
   *(black) & *(black) & *(black) & *(black) & *(black) & *(black) & {} & {} & {} & {} & {} &
   {} & {} & {} & {} & {} & {} & {} & {} & {} & {} & {} & {} & {} &
   {} & {} & {} & {} & {} & {} & {} & {} & {} & {} & {} & {} & {} &
   {} & {} & {} \\
 {} & {} & {} & *(black) & {} & {} & {} & {} & {} & {} & {} & {} & {} &
   {} & {} & {} & {} & {} & {} & {} & *(black) & *(black) & *(black) & *(black) & *(black)
   & *(black) & *(black) & *(black) & *(black) & *(black) & *(black) & *(black) & {} & {} & {} & {} &
   {} & {} & {} & {} & {} & {} & {} & {} & {} & {} & {} & {} & {} &
   {} & {} & {} & {} & {} & {} & {} & {} & {} & {} & {} & {} & {} &
   {} & {} & {} & {} \\
 {} & {} & {} & {} & {} & {} & {} & {} & {} & {} & {} & {} & {} &
   {} & {} & {} & {} & {} & {} & {} & *(black) & *(black) & *(black) & *(black) & *(black)
   & *(black) & {} & {} & {} & {} & {} & {} & {} & {} & {} & {} & {} & {}
   & {} & {} & {} & {} & {} & {} & {} & {} & {} & {} & {} & {} & {} &
   {} & {} & {} & {} & {} & {} & {} & {} & {} & {} & {} & {} & {} &
   {} & {} \\
 {} & {} & {} & {} & {} & {} & {} & {} & {} & {} & {} & {} & {} &
   {} & {} & {} & {} & {} & {} & {} & *(black) & *(black) & *(black) & *(black) & *(black)
   & *(black) & {} & {} & {} & {} & {} & {} & {} & {} & {} & {} & {} & {}
   & {} & {} & {} & {} & {} & {} & {} & {} & {} & {} & {} & {} & {} &
   {} & {} & {} & {} & {} & {} & {} & {} & {} & {} & {} & {} & {} &
   {} & {} \\
 {} & {} & {} & {} & *(black) & {} & {} & {} & {} & {} & {} & {} & {} &
   {} & {} & {} & *(black) & *(black) & *(black) & *(black) & *(black) & *(black) & *(black) & *(black) &
   *(black) & *(black) & {} & {} & {} & {} & {} & {} & {} & {} & {} & {} &
   {} & {} & {} & {} & {} & {} & {} & {} & {} & {} & {} & {} & {} &
   {} & {} & {} & {} & {} & {} & {} & {} & {} & {} & {} & {} & {} &
   {} & {} & {} & {} \\
 {} & {} & {} & {} & {} & {} & {} & {} & {} & {} & {} & {} & {} &
   {} & {} & {} & *(black) & *(black) & *(black) & *(black) & *(black) & {} & {} & {} & {}
   & {} & {} & {} & {} & {} & {} & {} & {} & {} & {} & {} & {} & {} &
   {} & {} & {} & {} & {} & {} & {} & {} & {} & {} & {} & {} & {} &
   {} & {} & {} & {} & {} & {} & {} & {} & {} & {} & {} & {} & {} &
   {} & {} \\
 {} & {} & {} & {} & {} & {} & {} & {} & {} & {} & {} & {} & {} &
   {} & {} & {} & *(black) & *(black) & *(black) & *(black) & *(black) & {} & {} & {} & {}
   & {} & {} & {} & {} & {} & {} & {} & {} & {} & {} & {} & {} & {} &
   {} & {} & {} & {} & {} & {} & {} & {} & {} & {} & {} & {} & {} &
   {} & {} & {} & {} & {} & {} & {} & {} & {} & {} & {} & {} & {} &
   {} & {} \\
 {} & {} & {} & {} & {} & {} & {} & {} & {} & {} & {} & {} & {} &
   {} & {} & {} & *(black) & *(black) & *(black) & *(black) & *(black) & {} & {} & {} & {}
   & {} & {} & {} & {} & {} & {} & {} & {} & {} & {} & {} & {} & {} &
   {} & {} & {} & {} & {} & {} & {} & {} & {} & {} & {} & {} & {} &
   {} & {} & {} & {} & {} & {} & {} & {} & {} & {} & {} & {} & {} &
   {} & {} \\
 {} & {} & {} & {} & {} & *(black) & {} & {} & {} & {} & {} & {} & {} &
   *(black) & *(black) & *(black) & *(black) & *(black) & *(black) & *(black) & *(black) & {} & {} & {} &
   {} & {} & {} & {} & {} & {} & {} & {} & {} & {} & {} & {} & {} &
   {} & {} & {} & {} & {} & {} & {} & {} & {} & {} & {} & {} & {} &
   {} & {} & {} & {} & {} & {} & {} & {} & {} & {} & {} & {} & {} &
   {} & {} & {} \\
 {} & {} & {} & {} & {} & {} & {} & {} & {} & {} & {} & {} & {} &
   *(black) & *(black) & *(black) & *(black) & {} & {} & {} & {} & {} & {} & {} & {} &
   {} & {} & {} & {} & {} & {} & {} & {} & {} & {} & {} & {} & {} &
   {} & {} & {} & {} & {} & {} & {} & {} & {} & {} & {} & {} & {} &
   {} & {} & {} & {} & {} & {} & {} & {} & {} & {} & {} & {} & {} &
   {} & {} \\
 {} & {} & {} & {} & {} & {} & {} & {} & {} & {} & {} & {} & {} &
   *(black) & *(black) & *(black) & *(black) & {} & {} & {} & {} & {} & {} & {} & {} &
   {} & {} & {} & {} & {} & {} & {} & {} & {} & {} & {} & {} & {} &
   {} & {} & {} & {} & {} & {} & {} & {} & {} & {} & {} & {} & {} &
   {} & {} & {} & {} & {} & {} & {} & {} & {} & {} & {} & {} & {} &
   {} & {} \\
 {} & {} & {} & {} & {} & {} & {} & {} & {} & {} & {} & {} & {} &
   *(black) & *(black) & *(black) & *(black) & {} & {} & {} & {} & {} & {} & {} & {} &
   {} & {} & {} & {} & {} & {} & {} & {} & {} & {} & {} & {} & {} &
   {} & {} & {} & {} & {} & {} & {} & {} & {} & {} & {} & {} & {} &
   {} & {} & {} & {} & {} & {} & {} & {} & {} & {} & {} & {} & {} &
   {} & {} \\
 {} & {} & {} & {} & {} & {} & {} & {} & {} & {} & {} & {} & {} &
   *(black) & *(black) & *(black) & *(black) & {} & {} & {} & {} & {} & {} & {} & {} &
   {} & {} & {} & {} & {} & {} & {} & {} & {} & {} & {} & {} & {} &
   {} & {} & {} & {} & {} & {} & {} & {} & {} & {} & {} & {} & {} &
   {} & {} & {} & {} & {} & {} & {} & {} & {} & {} & {} & {} & {} &
   {} & {} \\
 {} & {} & {} & {} & {} & {} & *(black) & {} & {} & {} & {} & *(black) &
   *(black) & *(black) & *(black) & *(black) & *(black) & {} & {} & {} & {} & {} & {} & {}
   & {} & {} & {} & {} & {} & {} & {} & {} & {} & {} & {} & {} & {} &
   {} & {} & {} & {} & {} & {} & {} & {} & {} & {} & {} & {} & {} &
   {} & {} & {} & {} & {} & {} & {} & {} & {} & {} & {} & {} & {} &
   {} & {} & {} \\
 {} & {} & {} & {} & {} & {} & {} & {} & {} & {} & {} & *(black) & *(black)
   & *(black) & {} & {} & {} & {} & {} & {} & {} & {} & {} & {} & {} & {}
   & {} & {} & {} & {} & {} & {} & {} & {} & {} & {} & {} & {} & {} &
   {} & {} & {} & {} & {} & {} & {} & {} & {} & {} & {} & {} & {} &
   {} & {} & {} & {} & {} & {} & {} & {} & {} & {} & {} & {} & {} &
   {} \\
 {} & {} & {} & {} & {} & {} & {} & {} & {} & {} & {} & *(black) & *(black)
   & *(black) & {} & {} & {} & {} & {} & {} & {} & {} & {} & {} & {} & {}
   & {} & {} & {} & {} & {} & {} & {} & {} & {} & {} & {} & {} & {} &
   {} & {} & {} & {} & {} & {} & {} & {} & {} & {} & {} & {} & {} &
   {} & {} & {} & {} & {} & {} & {} & {} & {} & {} & {} & {} & {} &
   {} \\
 {} & {} & {} & {} & {} & {} & {} & {} & {} & {} & {} & *(black) & *(black)
   & *(black) & {} & {} & {} & {} & {} & {} & {} & {} & {} & {} & {} & {}
   & {} & {} & {} & {} & {} & {} & {} & {} & {} & {} & {} & {} & {} &
   {} & {} & {} & {} & {} & {} & {} & {} & {} & {} & {} & {} & {} &
   {} & {} & {} & {} & {} & {} & {} & {} & {} & {} & {} & {} & {} &
   {} \\
 {} & {} & {} & {} & {} & {} & {} & {} & {} & {} & {} & *(black) & *(black)
   & *(black) & {} & {} & {} & {} & {} & {} & {} & {} & {} & {} & {} & {}
   & {} & {} & {} & {} & {} & {} & {} & {} & {} & {} & {} & {} & {} &
   {} & {} & {} & {} & {} & {} & {} & {} & {} & {} & {} & {} & {} &
   {} & {} & {} & {} & {} & {} & {} & {} & {} & {} & {} & {} & {} &
   {} \\
 {} & {} & {} & {} & {} & {} & {} & {} & {} & {} & {} & *(black) & *(black)
   & *(black) & {} & {} & {} & {} & {} & {} & {} & {} & {} & {} & {} & {}
   & {} & {} & {} & {} & {} & {} & {} & {} & {} & {} & {} & {} & {} &
   {} & {} & {} & {} & {} & {} & {} & {} & {} & {} & {} & {} & {} &
   {} & {} & {} & {} & {} & {} & {} & {} & {} & {} & {} & {} & {} &
   {} \\
 {} & {} & {} & {} & {} & {} & {} & *(black) & {} & {} & {} & *(black) &
   *(black) & *(black) & {} & {} & {} & {} & {} & {} & {} & {} & {} & {} &
   {} & {} & {} & {} & {} & {} & {} & {} & {} & {} & {} & {} & {} &
   {} & {} & {} & {} & {} & {} & {} & {} & {} & {} & {} & {} & {} &
   {} & {} & {} & {} & {} & {} & {} & {} & {} & {} & {} & {} & {} &
   {} & {} & {} \\
 {} & {} & {} & {} & {} & {} & {} & {} & {} & {} & {} & *(black) & {} &
   {} & {} & {} & {} & {} & {} & {} & {} & {} & {} & {} & {} & {} &
   {} & {} & {} & {} & {} & {} & {} & {} & {} & {} & {} & {} & {} &
   {} & {} & {} & {} & {} & {} & {} & {} & {} & {} & {} & {} & {} &
   {} & {} & {} & {} & {} & {} & {} & {} & {} & {} & {} & {} & {} &
   {} \\
 {} & {} & {} & {} & {} & {} & {} & {} & {} & {} & {} & *(black) & {} &
   {} & {} & {} & {} & {} & {} & {} & {} & {} & {} & {} & {} & {} &
   {} & {} & {} & {} & {} & {} & {} & {} & {} & {} & {} & {} & {} &
   {} & {} & {} & {} & {} & {} & {} & {} & {} & {} & {} & {} & {} &
   {} & {} & {} & {} & {} & {} & {} & {} & {} & {} & {} & {} & {} &
   {} \\
 {} & {} & {} & {} & {} & {} & {} & {} & {} & {} & {} & *(black) & {} &
   {} & {} & {} & {} & {} & {} & {} & {} & {} & {} & {} & {} & {} &
   {} & {} & {} & {} & {} & {} & {} & {} & {} & {} & {} & {} & {} &
   {} & {} & {} & {} & {} & {} & {} & {} & {} & {} & {} & {} & {} &
   {} & {} & {} & {} & {} & {} & {} & {} & {} & {} & {} & {} & {} &
   {} \\
 {} & {} & {} & {} & {} & {} & {} & {} & {} & {} & {} & *(black) & {} &
   {} & {} & {} & {} & {} & {} & {} & {} & {} & {} & {} & {} & {} &
   {} & {} & {} & {} & {} & {} & {} & {} & {} & {} & {} & {} & {} &
   {} & {} & {} & {} & {} & {} & {} & {} & {} & {} & {} & {} & {} &
   {} & {} & {} & {} & {} & {} & {} & {} & {} & {} & {} & {} & {} &
   {} \\
 {} & {} & {} & {} & {} & {} & {} & {} & {} & {} & {} & *(black) & {} &
   {} & {} & {} & {} & {} & {} & {} & {} & {} & {} & {} & {} & {} &
   {} & {} & {} & {} & {} & {} & {} & {} & {} & {} & {} & {} & {} &
   {} & {} & {} & {} & {} & {} & {} & {} & {} & {} & {} & {} & {} &
   {} & {} & {} & {} & {} & {} & {} & {} & {} & {} & {} & {} & {} &
   {} \\
 {} & {} & {} & {} & {} & {} & {} & {} & {} & {} & {} & *(black) & {} &
   {} & {} & {} & {} & {} & {} & {} & {} & {} & {} & {} & {} & {} &
   {} & {} & {} & {} & {} & {} & {} & {} & {} & {} & {} & {} & {} &
   {} & {} & {} & {} & {} & {} & {} & {} & {} & {} & {} & {} & {} &
   {} & {} & {} & {} & {} & {} & {} & {} & {} & {} & {} & {} & {} &
   {} \\
 {} & {} & {} & {} & {} & {} & {} & {} & *(black) & {} & {} & *(black) & {}
   & {} & {} & {} & {} & {} & {} & {} & {} & {} & {} & {} & {} & {} &
   {} & {} & {} & {} & {} & {} & {} & {} & {} & {} & {} & {} & {} &
   {} & {} & {} & {} & {} & {} & {} & {} & {} & {} & {} & {} & {} &
   {} & {} & {} & {} & {} & {} & {} & {} & {} & {} & {} & {} & {} &
   {} \\
 {} & {} & {} & {} & {} & {} & {} & {} & {} & {} & {} & *(black) & {} &
   {} & {} & {} & {} & {} & {} & {} & {} & {} & {} & {} & {} & {} &
   {} & {} & {} & {} & {} & {} & {} & {} & {} & {} & {} & {} & {} &
   {} & {} & {} & {} & {} & {} & {} & {} & {} & {} & {} & {} & {} &
   {} & {} & {} & {} & {} & {} & {} & {} & {} & {} & {} & {} & {} &
   {} \\
 {} & {} & {} & {} & {} & {} & {} & {} & {} & {} & {} & *(black) & {} &
   {} & {} & {} & {} & {} & {} & {} & {} & {} & {} & {} & {} & {} &
   {} & {} & {} & {} & {} & {} & {} & {} & {} & {} & {} & {} & {} &
   {} & {} & {} & {} & {} & {} & {} & {} & {} & {} & {} & {} & {} &
   {} & {} & {} & {} & {} & {} & {} & {} & {} & {} & {} & {} & {} &
   {} \\
 {} & {} & {} & {} & {} & {} & {} & {} & {} & {} & {} & *(black) & {} &
   {} & {} & {} & {} & {} & {} & {} & {} & {} & {} & {} & {} & {} &
   {} & {} & {} & {} & {} & {} & {} & {} & {} & {} & {} & {} & {} &
   {} & {} & {} & {} & {} & {} & {} & {} & {} & {} & {} & {} & {} &
   {} & {} & {} & {} & {} & {} & {} & {} & {} & {} & {} & {} & {} &
   {} \\
 {} & {} & {} & {} & {} & {} & {} & {} & {} & {} & {} & *(black) & {} &
   {} & {} & {} & {} & {} & {} & {} & {} & {} & {} & {} & {} & {} &
   {} & {} & {} & {} & {} & {} & {} & {} & {} & {} & {} & {} & {} &
   {} & {} & {} & {} & {} & {} & {} & {} & {} & {} & {} & {} & {} &
   {} & {} & {} & {} & {} & {} & {} & {} & {} & {} & {} & {} & {} &
   {} \\
 {} & {} & {} & {} & {} & {} & {} & {} & {} & {} & {} & *(black) & {} &
   {} & {} & {} & {} & {} & {} & {} & {} & {} & {} & {} & {} & {} &
   {} & {} & {} & {} & {} & {} & {} & {} & {} & {} & {} & {} & {} &
   {} & {} & {} & {} & {} & {} & {} & {} & {} & {} & {} & {} & {} &
   {} & {} & {} & {} & {} & {} & {} & {} & {} & {} & {} & {} & {} &
   {} \\
 {} & {} & {} & {} & {} & {} & {} & {} & {} & {} & {} & *(black) & {} &
   {} & {} & {} & {} & {} & {} & {} & {} & {} & {} & {} & {} & {} &
   {} & {} & {} & {} & {} & {} & {} & {} & {} & {} & {} & {} & {} &
   {} & {} & {} & {} & {} & {} & {} & {} & {} & {} & {} & {} & {} &
   {} & {} & {} & {} & {} & {} & {} & {} & {} & {} & {} & {} & {} &
   {} \\
 {} & {} & {} & {} & {} & {} & {} & {} & {} & {} & {} & *(black) & {} &
   {} & {} & {} & {} & {} & {} & {} & {} & {} & {} & {} & {} & {} &
   {} & {} & {} & {} & {} & {} & {} & {} & {} & {} & {} & {} & {} &
   {} & {} & {} & {} & {} & {} & {} & {} & {} & {} & {} & {} & {} &
   {} & {} & {} & {} & {} & {} & {} & {} & {} & {} & {} & {} & {} &
   {} \\
 {} & {} & {} & {} & {} & {} & {} & {} & {} & *(black) & {} & *(black) & {}
   & {} & {} & {} & {} & {} & {} & {} & {} & {} & {} & {} & {} & {} &
   {} & {} & {} & {} & {} & {} & {} & {} & {} & {} & {} & {} & {} &
   {} & {} & {} & {} & {} & {} & {} & {} & {} & {} & {} & {} & {} &
   {} & {} & {} & {} & {} & {} & {} & {} & {} & {} & {} & {} & {} &
   {} \\
 {} & {} & {} & {} & {} & {} & {} & {} & {} & {} & {} & *(black) & {} &
   {} & {} & {} & {} & {} & {} & {} & {} & {} & {} & {} & {} & {} &
   {} & {} & {} & {} & {} & {} & {} & {} & {} & {} & {} & {} & {} &
   {} & {} & {} & {} & {} & {} & {} & {} & {} & {} & {} & {} & {} &
   {} & {} & {} & {} & {} & {} & {} & {} & {} & {} & {} & {} & {} &
   {} \\
 {} & {} & {} & {} & {} & {} & {} & {} & {} & {} & {} & *(black) & {} &
   {} & {} & {} & {} & {} & {} & {} & {} & {} & {} & {} & {} & {} &
   {} & {} & {} & {} & {} & {} & {} & {} & {} & {} & {} & {} & {} &
   {} & {} & {} & {} & {} & {} & {} & {} & {} & {} & {} & {} & {} &
   {} & {} & {} & {} & {} & {} & {} & {} & {} & {} & {} & {} & {} &
   {} \\
 {} & {} & {} & {} & {} & {} & {} & {} & {} & {} & {} & *(black) & {} &
   {} & {} & {} & {} & {} & {} & {} & {} & {} & {} & {} & {} & {} &
   {} & {} & {} & {} & {} & {} & {} & {} & {} & {} & {} & {} & {} &
   {} & {} & {} & {} & {} & {} & {} & {} & {} & {} & {} & {} & {} &
   {} & {} & {} & {} & {} & {} & {} & {} & {} & {} & {} & {} & {} &
   {} \\
 {} & {} & {} & {} & {} & {} & {} & {} & {} & {} & {} & *(black) & {} &
   {} & {} & {} & {} & {} & {} & {} & {} & {} & {} & {} & {} & {} &
   {} & {} & {} & {} & {} & {} & {} & {} & {} & {} & {} & {} & {} &
   {} & {} & {} & {} & {} & {} & {} & {} & {} & {} & {} & {} & {} &
   {} & {} & {} & {} & {} & {} & {} & {} & {} & {} & {} & {} & {} &
   {} \\
 {} & {} & {} & {} & {} & {} & {} & {} & {} & {} & {} & *(black) & {} &
   {} & {} & {} & {} & {} & {} & {} & {} & {} & {} & {} & {} & {} &
   {} & {} & {} & {} & {} & {} & {} & {} & {} & {} & {} & {} & {} &
   {} & {} & {} & {} & {} & {} & {} & {} & {} & {} & {} & {} & {} &
   {} & {} & {} & {} & {} & {} & {} & {} & {} & {} & {} & {} & {} &
   {} \\
 {} & {} & {} & {} & {} & {} & {} & {} & {} & {} & {} & *(black) & {} &
   {} & {} & {} & {} & {} & {} & {} & {} & {} & {} & {} & {} & {} &
   {} & {} & {} & {} & {} & {} & {} & {} & {} & {} & {} & {} & {} &
   {} & {} & {} & {} & {} & {} & {} & {} & {} & {} & {} & {} & {} &
   {} & {} & {} & {} & {} & {} & {} & {} & {} & {} & {} & {} & {} &
   {} \\
 {} & {} & {} & {} & {} & {} & {} & {} & {} & {} & {} & *(black) & {} &
   {} & {} & {} & {} & {} & {} & {} & {} & {} & {} & {} & {} & {} &
   {} & {} & {} & {} & {} & {} & {} & {} & {} & {} & {} & {} & {} &
   {} & {} & {} & {} & {} & {} & {} & {} & {} & {} & {} & {} & {} &
   {} & {} & {} & {} & {} & {} & {} & {} & {} & {} & {} & {} & {} &
   {} \\
 {} & {} & {} & {} & {} & {} & {} & {} & {} & {} & {} & *(black) & {} &
   {} & {} & {} & {} & {} & {} & {} & {} & {} & {} & {} & {} & {} &
   {} & {} & {} & {} & {} & {} & {} & {} & {} & {} & {} & {} & {} &
   {} & {} & {} & {} & {} & {} & {} & {} & {} & {} & {} & {} & {} &
   {} & {} & {} & {} & {} & {} & {} & {} & {} & {} & {} & {} & {} &
   {} \\
 {} & {} & {} & {} & {} & {} & {} & {} & {} & {} & *(black) & *(black) & {}
   & {} & {} & {} & {} & {} & {} & {} & {} & {} & {} & {} & {} & {} &
   {} & {} & {} & {} & {} & {} & {} & {} & {} & {} & {} & {} & {} &
   {} & {} & {} & {} & {} & {} & {} & {} & {} & {} & {} & {} & {} &
   {} & {} & {} & {} & {} & {} & {} & {} & {} & {} & {} & {} & {} &
   {} \\
 {} & {} & {} & {} & {} & {} & {} & {} & {} & {} & {} & *(black) & {} &
   {} & {} & {} & {} & {} & {} & {} & {} & {} & {} & {} & {} & {} &
   {} & {} & {} & {} & {} & {} & {} & {} & {} & {} & {} & {} & {} &
   {} & {} & {} & {} & {} & {} & {} & {} & {} & {} & {} & {} & {} &
   {} & {} & {} & {} & {} & {} & {} & {} & {} & {} & {} & {} & {} &
   {} \\
 {} & {} & {} & {} & {} & {} & {} & {} & {} & {} & {} & *(black) & {} &
   {} & {} & {} & {} & {} & {} & {} & {} & {} & {} & {} & {} & {} &
   {} & {} & {} & {} & {} & {} & {} & {} & {} & {} & {} & {} & {} &
   {} & {} & {} & {} & {} & {} & {} & {} & {} & {} & {} & {} & {} &
   {} & {} & {} & {} & {} & {} & {} & {} & {} & {} & {} & {} & {} &
   {} \\
 {} & {} & {} & {} & {} & {} & {} & {} & {} & {} & {} & *(black) & {} &
   {} & {} & {} & {} & {} & {} & {} & {} & {} & {} & {} & {} & {} &
   {} & {} & {} & {} & {} & {} & {} & {} & {} & {} & {} & {} & {} &
   {} & {} & {} & {} & {} & {} & {} & {} & {} & {} & {} & {} & {} &
   {} & {} & {} & {} & {} & {} & {} & {} & {} & {} & {} & {} & {} &
   {} \\
 {} & {} & {} & {} & {} & {} & {} & {} & {} & {} & {} & *(black) & {} &
   {} & {} & {} & {} & {} & {} & {} & {} & {} & {} & {} & {} & {} &
   {} & {} & {} & {} & {} & {} & {} & {} & {} & {} & {} & {} & {} &
   {} & {} & {} & {} & {} & {} & {} & {} & {} & {} & {} & {} & {} &
   {} & {} & {} & {} & {} & {} & {} & {} & {} & {} & {} & {} & {} &
   {} \\
 {} & {} & {} & {} & {} & {} & {} & {} & {} & {} & {} & *(black) & {} &
   {} & {} & {} & {} & {} & {} & {} & {} & {} & {} & {} & {} & {} &
   {} & {} & {} & {} & {} & {} & {} & {} & {} & {} & {} & {} & {} &
   {} & {} & {} & {} & {} & {} & {} & {} & {} & {} & {} & {} & {} &
   {} & {} & {} & {} & {} & {} & {} & {} & {} & {} & {} & {} & {} &
   {} \\
 {} & {} & {} & {} & {} & {} & {} & {} & {} & {} & {} & *(black) & {} &
   {} & {} & {} & {} & {} & {} & {} & {} & {} & {} & {} & {} & {} &
   {} & {} & {} & {} & {} & {} & {} & {} & {} & {} & {} & {} & {} &
   {} & {} & {} & {} & {} & {} & {} & {} & {} & {} & {} & {} & {} &
   {} & {} & {} & {} & {} & {} & {} & {} & {} & {} & {} & {} & {} &
   {} \\
 {} & {} & {} & {} & {} & {} & {} & {} & {} & {} & {} & *(black) & {} &
   {} & {} & {} & {} & {} & {} & {} & {} & {} & {} & {} & {} & {} &
   {} & {} & {} & {} & {} & {} & {} & {} & {} & {} & {} & {} & {} &
   {} & {} & {} & {} & {} & {} & {} & {} & {} & {} & {} & {} & {} &
   {} & {} & {} & {} & {} & {} & {} & {} & {} & {} & {} & {} & {} &
   {} \\
 {} & {} & {} & {} & {} & {} & {} & {} & {} & {} & {} & *(black) & {} &
   {} & {} & {} & {} & {} & {} & {} & {} & {} & {} & {} & {} & {} &
   {} & {} & {} & {} & {} & {} & {} & {} & {} & {} & {} & {} & {} &
   {} & {} & {} & {} & {} & {} & {} & {} & {} & {} & {} & {} & {} &
   {} & {} & {} & {} & {} & {} & {} & {} & {} & {} & {} & {} & {} &
   {} \\
 {} & {} & {} & {} & {} & {} & {} & {} & {} & {} & {} & *(black) & {} &
   {} & {} & {} & {} & {} & {} & {} & {} & {} & {} & {} & {} & {} &
   {} & {} & {} & {} & {} & {} & {} & {} & {} & {} & {} & {} & {} &
   {} & {} & {} & {} & {} & {} & {} & {} & {} & {} & {} & {} & {} &
   {} & {} & {} & {} & {} & {} & {} & {} & {} & {} & {} & {} & {} &
   {} \\
 {} & {} & {} & {} & {} & {} & {} & {} & {} & {} & {} & *(black) & {} &
   {} & {} & {} & {} & {} & {} & {} & {} & {} & {} & {} & {} & {} &
   {} & {} & {} & {} & {} & {} & {} & {} & {} & {} & {} & {} & {} &
   {} & {} & {} & {} & {} & {} & {} & {} & {} & {} & {} & {} & {} &
   {} & {} & {} & {} & {} & {} & {} & {} & {} & {} & {} & {} & {} &
   {} \\
 \end{ytableau}}
\end{center}
\caption{All pairs $(n,m)$ for which $\sg(K_{n,m}) = 12$.}  \label{fig:k12}
\end{figure}
\end{example}

\proof[Proof of Theorem~\ref{thm:sg=k}]
It is not difficult to see that $\sg(K_{n,m}) = 2$ if and only if $(n,m) \in \{ (1,1), (1,2), (2,1) \}$, and that $\sg(K_{2,2}) = 3$. So assume that $k \geq 3$ and $\max\{n,m\} \geq 3$.

The statement follows from the following (note that the sum $s_1 + s_2$ equals $k$ for every $(s_1, s_2)$ that appears below):
\begin{enumerate}
\setlength{\itemsep}{0pt}
	\item If $n \leq 3$ and $m = f(k, n) = k$,  then $(0,k)$ is an optimal solution.
	\item If $3 \leq n < k$ and $m = f(k, n)$,  then $(n, k-n)$ is an optimal solution.
	\item If $m \leq 3$ and $n = f(k, m) = k$,  then $(k,0)$ is an optimal solution.
	\item If $3 \leq m < k$ and $n = f(k, m)$,  then $(k-m, m)$ is an optimal solution.
	\item If $f(k, i-1) < m \leq f(k,i)$ and $f(k, k-i-1) < n \leq f(k, k-i)$ for $4 \leq i \leq k-4$, then $(i, k-i)$ is an optimal solution.
	\item If $f(k, i-1) < m \leq f(k,i)$ and $n = f(k, k-i-1)$ for $i \leq k-3$, then $(i, k-i)$ is an optimal solution.
	\item If $f(k, i-1) < m \leq f(k,i)$ and $n = f(k, k-i-1)$ for $i \leq k-4$, then $(i+1, k-i-1)$ is an optimal solution.
	\item If $f(k, i-1) < m \leq f(k,i)$ and $n = f(k, k-i-1)$ for $i \geq k-4$, then $(k,0)$ is an optimal solution.
	\item If $m =f(k, i-1)$ and $f(k, k-i-1) < n \leq f(k,k-i)$ for $i \geq 3$, then $(i, k-i)$ is an optimal solution.
	\item If $m =f(k, i-1)$ and $f(k, k-i-1) < n \leq f(k,k-i)$ for $i \geq 4$, then $(i-1,k-i+1)$ is an optimal solution.
	\item If $m =f(k, i-1)$ and $f(k, k-i-1) < n \leq f(k,k-i)$ for $i \leq 4$, then $(0,k)$ is an optimal solution.
	\item If $m = f(k,i-1)$ and $n = f(k, k-i-1)$ for $i \leq 4$, then $(0,k)$ is an optimal solution.
	\item If $m = f(k,i-1)$ and $n = f(k, k-i-1)$ for $2 \leq i \leq k-4$, then $(i+1, k-i-1)$ is an optimal solution.
	\item If $m = f(k,i-1)$ and $n = f(k, k-i-1)$ for $3 \leq i \leq k-3$, then $(i, k-i)$ is an optimal solution.
	\item If $m = f(k,i-1)$ and $n = f(k, k-i-1)$ for $4 \leq i \leq k-2$, then $(i-1, k-i+1)$ is an optimal solution.
	\item If $m = f(k,i-1)$ and $n = f(k, k-i-1)$ for $i \geq k-4$, then $(k,0)$ is an optimal solution.
\end{enumerate}

It is easy to see that the above solutions give rise to the strong geodetic sets of size $k$. For example, in the first case, the part of the bipartition of size $m$ is a strong geodetic set with parameters $(0,k)$. What remains to be proved is $\sg(K_{n,m}) \geq k$ for each case. This can be shown by a simple case analysis. As the reasoning is similar in all cases, we demonstrate only two of them. Let $X$ be the part of the bipartition of size $n$ and $Y$ the part of size $m$. Also, let $S = S_1 \cup S_2$, where $S_1 \subseteq X$, $S_2 \subseteq Y$, be some strong geodetic set.

\begin{enumerate}
	\item[2.] In this case we have $k > n \geq 3$ and $m = k-1 + \binom{n-1}{2} = k - n + \binom{n}{2}$. If $S_1 = X$, then geodesics between these vertices cover at most $\binom{n}{2}$ vertices in $Y$, so at least $k-n$ vertices in $Y$ must also lie in a strong geodetic set. Hence, $|S| \geq n - (k-n) = k$.

	If $S_1 \neq X$, geodesics between these vertices cover at most $\binom{n-1}{2}$ vertices in $Y$, so at least $k-1$ vertices from $Y$ must lie in a strong geodetic set. Hence, $|S| \geq |S_1| + (k-1)$. If $S_1 \neq \emptyset$ or $|S_2| \geq k$, we have $|S| \geq k$. Otherwise, $S = S_2$ and contains exactly $k-1$ vertices. But then the remaining vertices in $Y$ are not covered.

	\item[5.] We can write
	$$n = k-1 + \binom{k-i-2}{2} + l,\ l \in \{1, \ldots, k-i-2\}\, ,$$
	$$m = k-1 + \binom{i-2}{2} + j,\ j \in \{1, \ldots, i-2\}\, .$$
	Suppose $|S| \leq k-1$.
	If $|S_1| \leq i-2$, these vertices cover at most $\binom{i-2}{2}$  vertices in $X$, thus at least $k$ vertices remain uncovered and $|S| \geq k$. Hence, $|S_1| \geq i-1$. Similarly, $|S_2| \geq k-i-1$.

	If $|S_1| = i-1$, then $\binom{i-1}{2}$ vertices in $Y$ are covered. As $k + j - i+1$ are left uncovered, it holds that $|S_2| \geq k-i+2$ and thus $|S| \geq k+1$.

	If $|S_2| = k-i-1$, then $\binom{k-i-1}{2}$ vertices in $X$ are covered. As $l+i+1$ are left uncovered, it holds that $|S_1| \geq i+2$ and thus $|S| \geq k+1$.

	Hence $|S_1| \geq i$ and $|S_2| \geq k-i$ and thus $|S| \geq k$.
\end{enumerate}

Note that all different optimal solutions are described above, hence some of the conditions could be merged.
\qed

The first condition from Theorem~\ref{thm:sg=k} can be simplified as follows.

\begin{corollary}
	\label{cor:bipartite}
	If $n \geq 3$ and $m > \binom{n}{2}$, then $\sg(K_{n,m}) = m + 1 - \binom{n-1}{2}$. If $n \leq 3$ and $m > n$, then $\sg(K_{n,m}) = m$.
\end{corollary}

When $m \leq \binom n 2$, Theorem~\ref{thm:sg=k} is harder to apply. Note, however, that the theorem suggests that $m$ is approximately equal to $k-1+\binom{i-1}2$, and $n$ is approximately equal to $k-1+\binom{k-i-1}2$. Furthermore, note that we can rewrite the system of equations (with known $m,n$ and variables $k,i$) $m = k-1+\binom{i-1}2$, $n = k-1+\binom{k-i-1}2$ as a polynomial equation of degree $4$ for $k$ (say by subtracting the two equations, solving for $i$, and plugging the result into one of the equations), and solve it explicitly. It seems that one of the four solutions is always very close to $\sg(K_{m,n})$. Denote the minimal distance between $\sg(K_{m,n})$ and a solution $k$ of $m = k-1+\binom{i-1}2$, $n = k-1+\binom{k-i-1}2$ by $e(m,n)$. Then our data indicates the following:

\begin{table}[!!h]
	\centering
	\begin{tabular}{|c||*{15}{c|}}\hline
		$n$ & 10 & 100 & 1000 & 10000 & 100000 \\ \hline
		$\max\{e(m,n) \colon n \leq m \leq \binom n 2\}$ & $1.094$ & $1.774$ & $1.941$ & $1.983$ &  $1.995$ \\ \hline
	\end{tabular}
\end{table}

\medskip

We conjecture the following.

\begin{conjecture}
 If $n \leq m \leq \binom n 2$, then $e(m,n) < 2$.
\end{conjecture}

If the conjecture is true, $\sg(K_{m,n})$ is among the (at most $16$) positive integers that are at distance $< 2$ from one of the four solutions of the system $m = k-1+\binom{i-1}2$, $n = k-1+\binom{k-i-1}2$. For each of these (at most) $16$ candidates, there are at most three (consecutive) $i$'s for which $f(k,i-1) \leq m \leq f(k,i)$, found easily by solving the quadratic equation $m = k - 1 + \binom{i-1}2$. For each such $i$, check if $f(k,k-i-1) \leq n \leq f(k,k-i)$. This allows for computation of $\sg(K_{m,n})$ with a constant number of operations.

\medskip

In the rest of this section we discuss the asymptotic behavior  of the strong geodetic problem on complete bipartite graphs.

\begin{theorem}
	\label{thm:asymptotic}
	Assume that $m \geq n \geq 2$, $m = \alpha n^\beta + \gamma n^{\beta - 1} + O(n^{\beta - 2})$. Then
	\[\sg(K_{n,m}) =  \begin{cases}
	  \alpha n^\beta + O(n^{\beta-1}), & m = \alpha n^\beta + O(n^{\beta-1})\mbox{ for } \beta > 2 \\
	  (\alpha-\frac 1 2) n^2 + O(n), & m = \alpha n^2 + O(n)\mbox{ for } \alpha > 1/2 \\
	  (\gamma + \frac 3 2) n + O(1), & m =n^2/2 + \gamma n + O(1) \mbox{ for } \gamma > -1/2 \\
	  n + O(1), & m =n^2/2 + \gamma n + O(1) \mbox{ for } \gamma \leq -1/2 \\
	  \sqrt{2\alpha} \cdot n + O(1), & m = \alpha n^2 + O(n)\mbox{ for } 0 < \alpha < 1/2 \\
	  \sqrt{2\alpha} \cdot n^{\beta/2} + O(1), & m =  \alpha n^\beta + O(n^{\beta-1})\mbox{ for } 1 < \beta < 2 \\
	  \sqrt{2}(1+\sqrt\alpha) n^{1/2} + O(1), & m = \alpha n + O(1)\mbox{ for } \alpha \geq 1
	\end{cases}\]
\end{theorem}
\begin{proof}[Sketch of proof]
	The first three cases follow from the last corollary. Indeed, in each of these cases, $m > \binom{n}2$ for $n$ large enough, so $\sg(K_{m,n}) = m + 1 - \binom{n-1}2$; which gives the stated result. For example, when $m = n^2/2 + \gamma n + O(1)$ for $\gamma > -1/2$, we get $\sg(K_{m,n}) = n^2/2 + \gamma n + O(1) + 1 - \frac{n^2-3n+2}{2} = (\gamma+3/2)n + O(1)$.\\
  Now assume that $m = n^2/2 + \gamma n + O(1)$, $\gamma \leq -1/2$. Assume first that there exists a non-negative integer $c$ so that $-\gamma - 5/2 < c + \lceil (1+\sqrt{8c+9})/2 \rceil < -\gamma - 3/2$. Denote $\lceil (1+\sqrt{8c+9})/2 \rceil$ by $d$, and define $k = k(n) = n - c$ and $i = i(n) = k - d - 1$. Let us check that for $n$ large enough,
  $$f(k, i-1) \leq m \leq f(k, i), \qquad f(k, k-i-1) \leq n \leq f(k, k-i).$$
  Indeed, we have
  \begin{multline*}
  f(k,i) = k - 1 + \binom{i-1}2 = n - c - 1 + \binom{n - c - d - 2}{2} \\
  = n - c - 1 + \frac{(n - c - d - 2)(n - c - d - 3)}{2} \\
  = n^2/2 + n (1 - (c + d + 5/2)) + O(1) > m,
  \end{multline*}
  \begin{multline*}
  f(k,i-1) = k - 1 + \binom{i-2}2 = n - c - 1 + \binom{n - c - d - 3}{2} \\
  = n - c - 1 + \frac{(n - c - d - 3)(n - c - d - 4)}{2} \\
  = n^2/2 + n (1 - (c + d + 7/2)) + O(1) < m.
  \end{multline*}
  On the other hand, we claim that $d$ is the smallest non-negative integer so that $\binom{d}2 \geq c + 1$. Indeed, if $\binom x 2 \geq c + 1$ for $x$, then $x^2 - x - 2c - 2 \geq 0$ and therefore $x \geq (1 + \sqrt{1+4(2c+2)})/2$. If $x$ is an integer, then $x \geq d$. It follows that $\binom{d-1}2 < c + 1 \leq \binom{d}2$. Then
  $$f(k,k - i - 1) = k - 1 + \binom{k - i - 2}2 = n - c - 1 + \binom{d - 1}{2} < n$$
  and
  $$f(k,k - i) = k - 1 + \binom{k - i - 1}2 = n - c - 1 + \binom{d}{2} \geq n.$$
  This proves that $\sg(m,n) = k$.\\
  If such $c$ does not exist, there are two options. One is that the sequence $(c +  \lceil (1+\sqrt{8c+9})/2 \rceil)_{c=0}^\infty$, which contains all natural numbers except $\binom j 2$, $j \geq 0$, skips the interval $[-\gamma - 5/2,-\gamma-3/2]$. Take $c$ so that
  $$c + \lceil (1+\sqrt{8c+9})/2 \rceil < -\gamma - 5/2 < -\gamma - 3/2 < c + 1 +\lceil (1+\sqrt{8c+17})/2 \rceil.$$
  It turns out that in that case, $8c+9$ is the square of an odd integer, say $8c + 9 = (2d-3)^2$ (i.e.~$c = (d^2-3j)/2$ and $(1+\sqrt{8c+9})/2 = d - 1$). Now take $k = k(n) = n - c$ and $i = k - d - 1$. Then $f(k,i-1) \leq m \leq f(k,i)$ and $f(k,k-i-1) \leq n \leq f(k,k-i)$ for $n$ large enough, so $\sg(K_{n,m}) = k$.\\
  Another case is that $\gamma$ is half a negative integer. A similar analysis holds, with some weak inequalities replacing strong inequalities, and dealing separately with cases $m \geq n^2/2 + \gamma n$ and $m < n^2/2 + \gamma n$. We leave this as an exercise for the reader.\\
  We sketch the proof for the last three cases. If $m \sim \alpha n^2$ for $0 < \alpha < 1/2$, take $k \sim \sqrt{2\alpha} \, n$ and $k - i \sim \sqrt{2(1-\sqrt{2\alpha})} \sqrt n$ (in particular, $i \sim \sqrt{2\alpha} \, n$). Then
  $$f(k,i) \sim k + \binom{i} 2 \sim \sqrt{2\alpha} \, n + \frac{i^2}{2} \sim \sqrt{2\alpha} \, n + \alpha n^2 \sim \alpha n^2$$
  and
  $$f(k,k-i) \sim k + \binom{k-i} 2 \sim \sqrt{2\alpha} \, n + (1-\sqrt{2\alpha}) n \sim n.$$
  We can adapt this to show that $\sg(K_{n,m}) \sim \sqrt{2\alpha} \,  n$.\\
  If $m \sim \alpha n^\beta$ for $\alpha > 0$ and $1 < \beta < 2$, we can repeat the previous calculation with $k \sim \sqrt{2\alpha} n^{\beta/2}$ and $k - i \sim \sqrt{2n}$, and if $m \sim \alpha n$, with $k \sim \sqrt 2(1+\sqrt \alpha) \sqrt n$ and $i \sim \sqrt{2 \alpha n}$.
\end{proof}

\section{On complete multipartite graphs}
\label{sec:complete_multipartite}

The optimization problem~\eqref{optimizationProblem} can be generalized to complete multipartite graphs. However, solving such a program seems rather difficult. Hence, we present an approximate program which gives a nice lower bound for the strong geodetic number of a complete multipartite graph. If $i$ vertices from one part are in a strong geodetic set, geodesics between them cover at most $\binom{i}{2}$ other vertices. In the following, we do not take into account the condition that they can only cover vertices in other parts, and that the number of selected vertices must be an integer. Let $G$ be a complete multipartite graph corresponding to the partition $\pi = \langle 1^{m_1}, \ldots, k^{m_k} \rangle$ and let $a_{ij}$ denote the number of parts of size $j$ with exactly $i$ vertices in the strong geodetic set. Thus we must have $\sum_{i = 0}^j a_{ij} = m_j$ and $\sum_{j=1}^k \sum_{i=0}^j \binom{i}{2} a_{ij} \geq \sum_{j = 1}^k \sum_{i = 0}^j (j-i) a_{ij}$. The second condition simplifies to $\sum_{j=1}^k \sum_{i=1}^j \binom{i+1}{2} a_{ij} \geq \sum_{j = 1}^k \sum_{i = 0}^j j a_{ij} = \sum_{j = 1}^k j m_j = n$. As $a_{0j}$'s do not appear in it anymore, we also simplify the first condition to $\sum_{i = 1}^j a_{ij} \leq m_j$ and get

\begin{align}
\label{optimizationProblemMulti}
\begin{split}
\min \quad & \sum_{j=1}^k \sum_{i=1}^j i a_{ij} \\
\text{subject to: } & \sum_{i = 1}^j a_{ij} \leq m_j\\
&  \sum_{j=1}^k \sum_{i=1}^j \binom{i+1}{2} a_{ij} \geq n\\
&  0 \leq a_{ij} \leq m_j
\end{split}
\end{align}

As the sequence $\binom{k}{2} - k$ is increasing for $k \geq 3$, it is better to select more vertices in a bigger part. Hence, the optimal solution is

\begin{align*}
a_{k,k} & =  m_k \\
a_{k-1,k-1} & =  m_{k-1} \\
 & \vdots  \\
a_{l+1,l+1} & =  m_{l+1} \\
a_{l,l} & =  \frac{l m_l + \cdots + 1 m_1 - \binom{k}{2} m_k - \cdots - \binom{l+1}{2} m_{l+1}}{\binom{l+1}{2}}\,,
\end{align*}
where $l$ is the smallest positive integer such that $\binom{k+1}{2} m_k + \cdots + \binom{l+2}{2} m_{l+1} \leq k m_k + \cdots + 1 m_1 = |V(K_{\langle 1^{m_1}, \ldots, k^{m_k} \rangle})|$, which is equivalent to $l m_l + \cdots + 1 m_1 \geq \binom{k}{2} m_k + \cdots + \binom{l+1}{2} m_{l+1}$, and
$$\sg(K_{\langle 1^{m_1}, \ldots, k^{m_k} \rangle}) \geq \left \lceil k m_k + \cdots + (l+1) m_{l+1} + \frac{l m_l + \cdots + m_1 - \binom{k}{2} m_k - \cdots - \binom{l+1}{2} m_{l+1}}{\frac{l+1}{2}} \right \rceil.$$

The result is particularly interesting in the case when $\pi = \langle k^m \rangle$, i.e.\ when we observe a multipartite graph with $m$ parts of size $k$, as we get $l = k$ and $$\sg(K_{\langle k^m \rangle}) \geq \left \lceil \frac{2 k m}{k + 1} \right \rceil \, .$$

On the other hand, considering a strong geodetic set consisting only of the whole parts of the bipartition yields an upper bound. At least $l \in \Z$, where $l (k + \binom{k}{2}) \geq m k$, parts must be in a strong geodetic set. Hence,
$$\sg(K_{\langle k^m \rangle}) \leq \left \lceil \frac{2 m}{k + 1} \right \rceil \cdot k \, .$$

This implies the following result.

\begin{proposition}
	\label{prop:km}
	If $k, n \in \mathbb{N}$ and $(k+1)|2m$, then $\sg(K_{\langle k^m \rangle}) = \frac{2 m k}{k + 1}$.
\end{proposition}

\section{Complexity results for multipartite graphs}
\label{sec:NP}

The strong geodetic problem can be naturally formed as a decision problem.

\vspace{2mm}
\textsc{Strong geodetic set}\\
Input: a graph $G$, an integer $k$\\
Question: does a graph $G$ have a strong geodetic set of size at most $k$?
\vspace{2mm}

The strong geodetic problem on general graphs is known to be NP-complete~\cite{MaKl16a}. In the following we prove that it is also NP-complete on multipartite graphs.

The reduction uses the dominating set problem. Recall that a set $D \subseteq V(G)$ is a dominating set in the graph $G$ if every vertex in $V(G) - D$ has a neighbor in $D$.

\vspace{2mm}
\textsc{Dominating set}\\
Input: a graph $G$, an integer $k$\\
Question: does a graph $G$ have a dominating set of size at most $k$?
\vspace{2mm}

The dominating set problem is known to be NP-complete on bipartite graphs~\cite{Liedloff}, hence it is also NP-complete on multipartite graphs. The idea of the following proof is similar to the proof that the ordinary geodetic problem restricted to chordal bipartite graphs is NP-complete~\cite{dourado}.

\begin{theorem}
	\label{thm:NPC}
	\textsc{Strong geodetic set} restricted to bipartite graphs is NP-complete.
\end{theorem}

\proof
To prove NP-completeness, we describe a polynomial reduction of \textsc{Dominating set} on bipartite graphs to \textsc{Strong geodetic set} on bipartite graphs. Let $(G,k)$ be an input for \textsc{Dominating set}, and $(X, Y)$ a bipartition of the graph $G$. Define a graph $G'$, $$V(G') = V(G) \cup \{ u_1, u_2 \} \cup \{ x' \; ; \; x \in X \} \cup \{ y' \; ; \; y \in Y \},$$
with the edges $E(G)$, $u_1 \sim u_2$, and $x \sim u_2 \sim x'$ for all $x \in X$, $y \sim u_1 \sim y'$ for all $y \in Y$. Define the sets $X' = X \cup \{ u_1 \} \cup \{ x' \; ; \; x \in X \}$, $Y' = Y \cup \{ u_2 \} \cup \{ y' \; ; \; y \in Y \}$,
and observe that $(X', Y')$ is a bipartition of the graph $G'$. Define the parameter $k' = k + |V(G)|$.

Suppose $D$ is a dominating set of the graph $G$ of size at most $k$.  Define $$D' = D \cup \{ x' \; ; \; x \in X \} \cup \{ y' \; ; \; y \in Y \}.$$ Notice that $|D'| \leq k'$. For each $x \in X \cap D$, fix geodesics
$x \sim y \sim u_1 \sim y', y \in N_G(x)$.
Similarly, for each $y \in Y \cap D$, fix $y \sim x \sim u_2 \sim x', x \in N_G(y)$.
As $D$ is a dominating set, these geodesics cover all vertices in $V(G)$. Additionally, fix geodesics $x \sim u_2 \sim x'$ for some $x \in X$, and $y \sim u_1 \sim y'$ for some $y \in Y$, to cover the vertices $u_1, u_2$. Hence, $D'$ is a strong geodetic set of the graph $G'$.

Conversely, suppose $D'$ is a strong geodetic set of $G'$ of size at most $k'$. Vertices $\{ x' \; ; \; x \in X \} \cup \{ y' \; ; \; y \in Y \}$ are all simplicial, hence they all belong to $D'$. Geodesics between them cannot cover any vertices in $V(G)$, thus $V(G) \cap D' \neq \emptyset$. Let $D = D' \cap V(G)$. Clearly, $|D| \leq k$. Consider $x \in V(G) - D$. Thus $x$ is an inner point of some $y,z$-geodesic. At most one of $y, z$ does not belong to $D$. The structure of the graph ensures that at least one of $y, z$ is a neighbor of $x$. Hence, $D$ is a dominating set of the graph $G$.
\qed

\begin{corollary}
	\label{cor:NPC}
	\textsc{Strong geodetic set} restricted to multipartite graphs is NP-complete.
\end{corollary}

In the following we consider the complexity of \textsc{Strong geodetic set} on \emph{complete} multipartite graphs. Proposition~\ref{prop:two} gives rise to the following algorithm.

Let $G$ be a graph and $(X_1, \ldots, X_r)$ its multipartition. Denote $n_i = |X_i|$, $i \in [r]$. For all $\{i, j\} \subseteq \binom{[r]}{2}$, for all subsets $R$ of $[r] - \{i,j\}$, for all $s_i \in \{0,\ldots,n_i\}$, for all $s_j \in \{0,\ldots,n_j\}$, set $S_i \subseteq X_i$ of size $s_i$, and $S_j \subseteq X_j$ of size $s_j$. Check if $S_i \cup S_j \cup \bigcup_{k \in R} X_k$ is a strong geodetic set for $G$. The answer is the size of the smallest strong geodetic set.

The time complexity of this algorithm is $O(n^2 r^2 2^r)$. This confirms the already known result that \textsc{Strong geodetic set} restricted to complete bipartite graphs is in $P$, which is an easy consequence of Theorem~\ref{thm:sg=k}. Moreover, it is now clear that the problem is solvable in quadratic time. The same holds for complete $r$-partite graphs (when $r$ is fixed). But for a general complete multipartite graph (when the size of the multipartition is part of the input), the algorithm tells us nothing about complexity.

But we also observe an analogy between the \textsc{Strong geodetic set} problem on complete multipartite graphs and the \textsc{Knapsack problem}, which is known to be NP-complete~\cite{NP}. Recall that in this problem, we are given a set of items with their weights and values, and we need to determine which items to put in a backpack, so that a total weight is smaller that a given bound and a total value is as large as possible. The approximate reduction from the \textsc{Strong geodetic set} on complete multipartite graphs to the \textsc{Knapsack problem} is the following. Let $(X_1, \ldots, X_r)$ be the parts of the complete multipartite graph. The items $x_1, \ldots, x_r$ represent those parts, a value if $x_i$ is $\binom{|X_i|}{2}$ and the weight is $|X_i|$. Thus selecting the items such that their total value is as large as possible and the total weight as small as possible, is almost the same as finding the smallest strong geodetic set of the complete multipartite graph (as Proposition~\ref{prop:two} states that at most two parts in the strong geodetic set are selected only partially). We were not able to find a reduction from the \textsc{Knapsack problem} to the \textsc{Strong geodetic set} on complete multipartite graphs. But due to the connection with the \textsc{Knapsack problem}, it seems that the problem is not polynomial. Hence we pose

\begin{conjecture}
	\label{conjecture}
	\textsc{Strong geodetic set} restricted to complete multipartite graphs is NP-complete.
\end{conjecture}

However, as already mentioned, determining the strong geodetic number of complete $r$-partite graphs for fixed $r$ is polynomial. Using a computer program (implemented in Mathematica) we derive the results shown in Table~\ref{tbl:values}.

\begin{table}[htb]
	\centering
	\begin{tabular}{ c || c || c | c || c | c | c || c | c | c | c | c }
		\hline
		$\pi$ & $\langle 1 \rangle$ & $\langle 2 \rangle$ & $\langle 1^2 \rangle$ & $\langle 3 \rangle$ & $\langle 1,2 \rangle$ & $\langle 1^3 \rangle$ & $\langle 4 \rangle$ & $\langle 1,3 \rangle$ & $\langle 2^2 \rangle$ & $\langle 1^2, 2 \rangle$ & $\langle 1^4 \rangle$ \\ \hline \hline
		$\sg(K_{\pi})$ & 1 & 2 & 2 & 3 & 2 & 3 & 4 & 3 & 3 & 3 & 4 \\
		\hline
	\end{tabular}

	\medskip

	\begin{tabular}{ c || c | c | c | c | c | c | c }
		\hline
		$\pi$ & $\langle 5 \rangle$ & $\langle 1,4 \rangle$ & $\langle 2,3 \rangle$ & $\langle 1^2, 3 \rangle$ & $\langle 1, 2^2 \rangle$ & $\langle 1^3, 2 \rangle$ & $\langle 1^5 \rangle$  \\ \hline \hline
		$\sg(K_{\pi})$ & 5 & 4 & 3 & 3 & 4 & 4 & 5  \\
		\hline
	\end{tabular}

	\medskip

	\begin{tabular}{ c || c | c | c | c | c | c | c | c | c | c | c }
		\hline
		$\pi$ & $\langle 6 \rangle$ & $\langle 1,5 \rangle$ & $\langle 2,4 \rangle$ & $\langle 1^2,4 \rangle$ & $\langle 3^2 \rangle$ & $\langle 1,2,3 \rangle$ & $\langle 1^3,3 \rangle$ & $\langle 2^3 \rangle$ & $\langle 1^2,2^2 \rangle$ & $\langle 1^4, 2 \rangle$ & $\langle 1^6 \rangle$ \\ \hline \hline
		$\sg(K_{\pi})$ & 6 & 5 & 4 & 4 & 3 & 3 & 3 & 4 & 4 & 5 & 6 \\
		\hline
	\end{tabular}

	\medskip

	\begin{tabular}{ c || c | c | c | c | c | c | c | c }
		\hline
		$\pi$ & $\langle 7 \rangle$ & $\langle 1,6 \rangle$ & $\langle 2,5 \rangle$ & $\langle 1^2,5 \rangle$ & $\langle 3,4 \rangle$ & $\langle 1,2,4 \rangle$ & $\langle 1^3,4 \rangle$ & $\langle 1,3^2 \rangle$ \\ \hline \hline
		$\sg(K_{\pi})$ & 7 & 6 & 5 & 5 & 4 & 4 & 4 & 4 \\
		\hline
	\end{tabular}

	\medskip

	\begin{tabular}{ c || c | c | c | c | c | c | c }
		\hline
		$\pi$ & $\langle 2^2,3 \rangle$ & $\langle 1^2,2,3 \rangle$ & $\langle 1^4,3 \rangle$ & $\langle 1,2^3 \rangle$ & $\langle 1^3,2^2 \rangle$ & $\langle 1^5,2 \rangle$ & $\langle 1^7 \rangle$ \\ \hline \hline
		$\sg(K_{\pi})$ & 4 & 4 & 4 & 5 & 5 & 6 & 7 \\
		\hline
	\end{tabular}

	\caption{The strong geodetic numbers for some small complete multipartite graphs.}
	\label{tbl:values}
\end{table}

\section*{Acknowledgments}
\label{sec:Acknowledgments}

The authors would like to thank Sandi Klav\v zar and Valentin Gledel for a number of helpful conversations and suggestions.


\end{document}